\documentclass{amsart}

\usepackage{amsmath, amssymb, graphicx, epsfig, verbatim}
\usepackage{epic}

\newtheorem{thm}{Theorem}[section]

\newtheorem{prop}[thm]{Proposition}

\newtheorem{conj}[thm]{Conjecture}
\newtheorem{exa}[thm]{Example}
\newtheorem{rem}[thm]{Remark}

\numberwithin{equation}{section}

\newcommand{\Z}{\mathbb Z}

\newcommand{\CP}{{\mathbb C}{\mathbb P}}

\begin{document}

\title{On negative spheres in elliptic surfaces}

\author{Andr\'{a}s I. Stipsicz}
\address{R\'enyi Institute of Mathematics\\
H-1053 Budapest\\ 
Re\'altanoda utca 13--15, Hungary}
\email{stipsicz.andras@renyi.hu}

\author{Zolt\'an Szab\'o}
\address{Department of Mathematics\\
Princeton University\\
 Princeton, NJ, 08544}
\email{szabo@math.princeton.edu}

\begin{abstract}
  Using elliptic fibrations with specific singular fibers, we find
  spheres with very negative self-intersections in elliptic surfaces
  and in their blow-ups.
\end{abstract}
\maketitle

\section{Introduction}
\label{sec:intro}

Suppose that $X$ is a simply connected, closed, smooth, oriented four-manifold.
As $H_2\cong \pi _2$ in this case, all second homology classes can be
represented by a smooth map $f\colon S^2\to X$, and in this dimension we can
assume that $f$ is an immersion. There are, however, serious restrictions
on the homology class if we demand $f$ to be an embedding, i.e. if
the homology class can be represented by an embedded sphere.
(We will call such a homology class \emph{spherical}.)

The question of which integers appear as self-intersections of
spherical classes has been studied for some time. For example, in
$S^2\times S^2$ (being diffeomorphic to the Hirzebruch surfaces
${\mathbb {F}}_n$ for any even $n$, hence containing spheres of
self-intersections $n$ and $-n$) we have only mild homological
constrains (namely the parity of the self-intersection).  A similar
argument shows spheres in ${\mathbb {CP}}^2\# {\overline {\CP}} ^2$
with arbirtarily odd negative and positive self-intersections, and
similar statements hold for the blow-ups of these manifolds (which are
diffeomorphic to ${\mathbb {CP}}^2\# n{\overline {\CP}} ^2$ with
$n\geq 2$).

If $X$ has nontrivial Seiberg-Witten invariants \cite{Mor, W} and
$b_2^+(X)>1$, then the possibilities are much more restrictive; for
example, the self-intersection of a non-torsion spherical homology
class in such a four-manifold must be negative.  It is an open
question, however, if the negative numbers appearing as
self-intersections of spheres in a four-manifold $X$ with $b_2^+(X)>1$
and non-trivial Seiberg-Witten invariants $SW_X$ form a bounded set.

For the $K3$ surface, the unique smooth four-manifold (up to
diffeomorphism) which is simply connected and admits a complex
structure with $c_1=0$, so far the largest (in absolute value)
negative self-intersection has been found in \cite{FM}: Finashin and
Mikhalkin showed that for any even $k$ between $-2$ and $-86$
(including the two endpoints) there is a spherical class with
self-intersection $k$. (As the $K3$ surface is spin,
self-intersections are even.)  It is still an open question whether
there are spherical classes in the $K3$ surface with self-intersection
$<-86$.

In this note we generalize the result of \cite{FM} to simply connected
elliptic surface which admit a section and have $b_2^+>1$ (these
manifolds are usually denoted by $E(n)$ with $n>1$, and in this
notation $K3$ is simply $E(2)$, cf. \cite{FrMo, GS}). We also extend
results to blow-ups of these elliptic surfaces.

\begin{thm} \label{thm:negspheres}
  The four-manifold $E(n)$ with $n\geq 2$ contains a sphere with
  self-intersection
  \begin{equation}\label{ed:sinter}
  s(n)= -44.2\cdot n + 0.8 \cdot (5-r)
  \end{equation}
  where $r\in \{ 0, 1,2,3, 4\}$ is the residue of $n$ mod 5.
\end{thm}
This result recovers the Finashin-Mikhalkin example for $n=2$ and
generalizes it to the further $E(n)$'s.

The above examples can be modified and adapted in the blown-up
elliptic surfaces.  In particular, we have

\begin{thm}\label{thm:negsphereinblowup}
  Let $X_{n,k}$ denote the $k$-fold blow-up of the elliptic surface
  $E(n)$ (once again, with $n\geq 2$). 
  Then $X_{n,k}$ contains a sphere of self-intersection $s(n)-5k$.
  \end{thm}
In some cases we find spheres in $X_{n,k}$ with self-intersection less
than the value given in the above theorem (see the examples at the end
of Section~\ref{sec:blowups}). Examining the examples we have found,
we arrived to the following conjecture:

\begin{conj}\label{conj:mostnegsphere}
  There is a universal constant $C$ with the following property.
  For a closed, oriented four-manifold $X$ with $b_2^+(X)>1$ and
  with nontrivial   Seiberg-Witten invariant $SW_X$, and 
  $S\subset X$  smoothly embedded sphere in $X$, 
  the self-intersection $[S]^2$ satisfies the inequality
  \[
    [S]^2\geq C\cdot b_2(X).
    \]
   \end{conj}
The best candidate for $C$ so far (based on examples in elliptic
surfaces and their blow-ups) is $C= -5$. Note that by \cite{taubes}
the condition $SW_X\neq 0$ on the Seiberg-Witten invariants is
satisfied by all symplectic four-manifolds with $b_2^+>1$.
    
\begin{rem}
  We can also consider locally flat maps $f\colon S^2\to X$, for which
  the situation is drastically different: as we disregard the smooth
  structure of $X$ in this way, it is not hard to see that once $X$ is
  indefinite, we can have arbitrarily large (in absolute value)
  positive and negative self-intersections of spheres. The argument
  rests on the fact that both $S^2$-bundles over $S^2$ ($S^2\times
  S^2$ and $\CP ^2 \# {\overline {\CP}}^2$) contain spheres of
  arbitrarily large positive and negative self-intersections, and
  indefinite simply connected topological manifolds admit one of the
  above bundles as direct summands. The question still remains
  interesting for definite four-manifolds, though. See
  \cite{ballinger} for self-intersections of smooth spheres in the
  definite smooth four-manifolds $\# n{\mathbb {CP}}^2$.
  \end{rem}

The question of minimal self-intersections of spheres can be extended
to arbitrary genus: fix a genus $g$ and consider those homology
classes in the closed, oriented, smooth four-manifold $X$ with
$b_2^+(X)>1$ (and with non-trivial $SW_X$) which can be represented by
a smoothly embedded genus-$g$ surface. The adjunction inequality
provides an upper bound for the self-intersection of such surfaces;
we did not find examples of arbitrary negative self-intersections
(once again, with fixed $g$).

{\bf {Acknowledgements}}: AS was partially supported by the
\emph{\'Elvonal (Frontier) grant} KKP126683 of the NKFIH.  ZSz was
partially supported by NSF Grant DMS-1904628. We would like to thank
William Ballinger for enlightening discussions.

\section{Ellitpic fibrations and monodromies}
\label{sec:elliptics}
Suppose that $X$ is a compact complex surface and $f\colon X\to
{\mathbb {CP}}^1$ is a holomorphic map.  The map $f$ is an
\emph{elliptic fibration} if for a generic $t\in {\mathbb {CP}}^1$ the
fiber $f^{-1}(t)$ over $t$ is an elliptic curve (topologically a
two-dimensional torus).  A \emph{section} of an elliptic fibration is
a map $p\colon {\mathbb {CP}}^1 \to X$ satisfying $f\circ p=id_{{\mathbb
    {CP}}^1}$.  A complex surface is an elliptic surface if it admits
an elliptic fibration.

Singular fibers in an elliptic fibration were classified by Kodaira
(see~\cite{HKK}). In the following we will not need all possible
singular elliptic fibers; we restrict our attention to those which
will be useful in proving our main results. Indeed, the singular
fibers ${\widetilde {E}}_j$ (with $j=6,7,8$) and the singular fiber
$I^*_0$ can be given by the plumbing diagrams of
Figure~\ref{fig:Singfibers}.

\begin{figure}
\centering
\includegraphics[width=12cm]{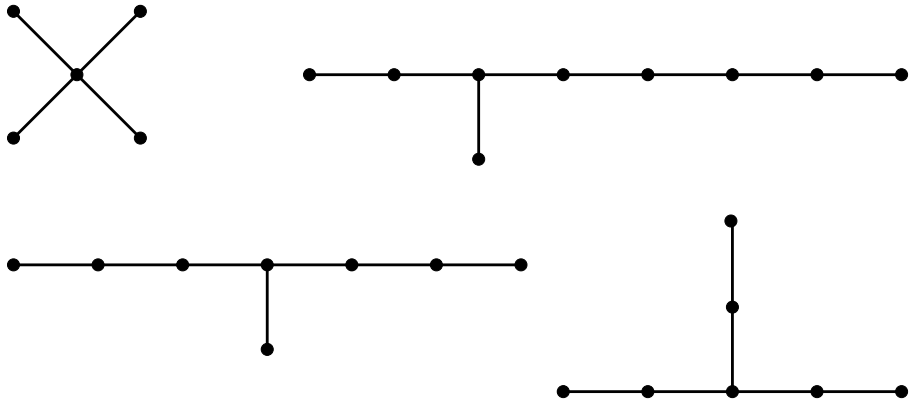}
\caption{ {\bf {The plumbing graphs of the singular fibers $I_0^*$ (upper left)
      and ${\widetilde {E}}_j$ (with $j+1$ vertices) for $j=6,7,8$
      in an elliptic fibration.}} In the plumbing graphs all vertices
  represent spheres and come with self-intersection $-2$. The monodromies are: $(ab)^3$ (for $I_o^*$),
  $(ab)^5$ (for ${\widetilde {E}}_8$), $(ab)^4a$ (for ${\widetilde {E}}_7$) and
$(ab)^4$ (for ${\widetilde {E}}_6$).}
\label{fig:Singfibers} 
\end{figure}

Elliptic surfaces were classified by Kodaira. Simply connected
elliptic surface admitting a section can be constructed as follows:
consider two cubic curves in the complex projective plane ${\mathbb
  {CP}}^2$ so that every intersection point is a smooth point of at
least one of the curves. Consider the pencil generated by these
curves. (If the two curves are given as zero sets of the cubic
polynomials $p_0$ and $p_1$, then the pencil is the family of curves
defined by the cubic polynomials $p_t=t_0p_0+t_1p_1$ with the
projective parameter $t=[t_0:t_1]\in {\mathbb {CP}}^1$.)  The
nine-fold (possibly infinitely close) blow-up of the pencil on
${\mathbb {CP}}^2$ provides an elliptic fibration on $E(1)={\mathbb
  {CP}}^2 \# 9{\overline {\mathbb {CP}}}^2$. Taking $n$-fold fiber
sums of this surface with itself, we get $E(n)$, and up to complex
deformation (according to Kodaira's classification) in this way we get
all simply connected elliptic surfaces admitting sections,
cf. also~\cite{GS}.  This construction does not provide all the
elliptic \emph{fibrations} on these smooth four-manifolds, though: not
every fibration on $E(n)$ can be described as a fiber sum of
fibrations on $E(1)$.

Consider a disk in ${\mathbb {CP}}^1$ containing some points with the
property that their inverse image is not a regular fiber (those are
called singular fibers). Assume furthermore that the boundary of the
disk does not contain any such points. Then the restriction of the
fibration to this boundary circle is a torus bundle over the circle,
which can be desribed by its monodromy, an element of the mapping
class group of the torus, isomorphic to the group ${\rm {SL}}_2 (\Z
)$. The monodromy of the simplest singular fiber (coming from a nodal curve)
corresponds to a right-handed Dehn twist along a simple closed curve
in the torus, while more complicated singular fibers (or more singular
fibers in a disk) give rise to more complicated words in 
${\rm {SL}}_2 (\Z )$.

The group ${\rm {SL}}_2 (\Z )$ admits a presentation as
\[
\langle a, b \mid aba=bab, (ab)^6=1\rangle ,
\]
where $a$ and $b$ can be chosen to be right-handed Dehn twists along
two simple closed curves in the torus intersecting each other
transversely in a single point.  Furthermore, the ${\widetilde
  {E}}_8$ fiber has monodromy $(ab)^5$, the ${\widetilde {E}}_6$ fiber
has monodromy $(ab)^4$, while the $I^*_0$ fiber has monodromy
$(ab)^3$. (More on these singular fibers see \cite[pages~35-36]{HKK}.)

For a generic choice of cubic curves, the blow-up of the corresponding
pencil provides the monodromy presentation $(ab)^6$ on $E(1)$.  For
$E(n)$ the corresponding presentation is $(ab)^{6n}$.

\begin{proof}[Proof of Theorem~\ref{thm:negspheres}]
  Suppose that $n=5k+r$, with $r\in \{ 0, \ldots , 4\}$.
  Then the monodromy $(ab)^{6n}$ is equal to $(ab)^{30k+6r}=
  ((ab)^5)^{6k}(ab)^{6r}$. Write $6r$ as
  \begin{itemize}
  \item 0 when $r=0$
  \item $3+3$ when $r=1$
  \item $5+4+3$ when $r=2$
  \item $5+5+5+3$ when $r=3$, and
  \item $5+5+5+5+4$ when  $r=4$.
  \end{itemize}
  With these decompositions, the monodromy presentation equips $E(n)$
  with a fibration containing $6k$ ${\widetilde {E}}_8$-fibers and
  some possible further ${\widetilde {E}}_8$-fibers (for $r=2,3,4$,
  corresponding to the 5's in the decomposition of $6r$) one or two
  $I^*_0$ fibers (corresponding to the 3's in the decomposition of
  $6r$) and possibly one further ${\widetilde {E}}_6$-fiber (for
  $r=2,4$). Together with a section of the fibration (which is a
  sphere of self-intersection $-n$, intersecting each fiber in a
  single point) these fibers give rise to a configuration of
  $(-2)$spheres (with the exception of the section, which has
  self-intersection $-n$) intersecting each other transversely
  according to a tree. Orienting them so that all intersections are
  negative (which is possible, as a tree is two-colorable), and then
  smoothing the intersections, we get a sphere with the desired
  self-intersection, concluding the proof.
\end{proof}
\begin{rem}
This construction essentially generalizes the proof given for the $K3$
surface in \cite{FM}, although they decomposed $12$ as $4+4+4$ (and
not as $5+4+3$ as above), receiving a more symmetric diagram in
\cite[Figure~3]{FM}.
\end{rem}

\section{Blown up elliptic surfaces}
\label{sec:blowups}
Suppose that the four-manifold $X$ contains a sphere of
self-intersection $-m$. As ${\overline {\CP}}^2$ contains a
$(-4)$-sphere, we can easily find in the $k$-fold blow-up $X\#
k{\overline {\CP}}^2$ a sphere of self-intersection $-m-4k$ by
tubing the spheres in the components of the connected sum decomposition
together.

In some cases, however,  this construction can be significantly improved.

\begin{prop}\label{prop:twosphereblow}
  Suppose that $X$ contains two spheres transversely intersecting each
  other in a single point, and having self-intersections $x$ and $y$.
  Then in the $k$-fold blow-up $X\# k{\overline {\mathbb {CP}}}^2$
  there is an embedded sphere with self-intersection $x+y-2-5k$.
\end{prop}
\begin{proof}
  Indeed, consider the two spheres in $X$ and blow up their (unique)
  intersection point. Together with the exceptional divisor we get now
  three embedded spheres of self-intersections $x-1$, $-1$ and $y-1$,
  intersecting each other along a linear graph. Blowing up edges of
  this graph (or, phrased differently, intersection points of the spheres),
  smoothing the intersection points (with appropriate orientations),
  a simple count verifies the claim.
  \end{proof}

\begin{proof}[Proof of Theorem~\ref{thm:negsphereinblowup}]
  In an elliptic surface we construct the sphere of self-intersection
  $s(n)$ by finding a tree of embedded spheres, and then smoothing
  their intersection points. By skipping one smoothing we get the
  configuration of two transversely intersecting spheres as demanded
  by Proposition~\ref{prop:twosphereblow}, hence the application of
  the proposition concludes the proof.
\end{proof}

For the blown-up elliptic surfaces we can use other singular fibers to
blow up, providing further negative spheres.  Indeed, a cups fiber
(also called type II fiber in the table of \cite[page 35]{HKK}) is a
singular sphere in an elliptic fibration, which admits one singular
point modelled by the cone on the trefoil knot.  The monodromy of a
cusp fiber is equal to $ab$.  We get a configuration of smooth curves
intersecting each other transversally (and without triple
intersections) after blowing up the cusp point three times.  In these
blow-ups the fiber will become a smooth sphere of self-intersection
$(-6)$, connected to a sphere of self-intersection $(-1)$, which is
connected to a $(-2)$- and a $(-3)$-sphere.  In a similar way, we can
use a type III fiber (with monodromy $aba$), which consists of two
$(-2)$-curves tangent to each other.  After two blow-ups this fiber
transforms to a configuration of a $(-1)$-sphere, intersected by two
$(-4)$-spheres and a $(-2)$-sphere. Sometimes a type IV fiber provides
the best result; this fiber consists of three $(-2)$-spheres passing
thorugh the same point, and has monodromy $(ab)^2$. Blowing up the
common intersection of the three spheres once, we get a configuration
of four spheres, a central $(-1)$-sphere intersected by three $(-3)$-spheres.

Another way to get an embedded sphere proceeds as follows. Suppose
that in the elliptic surface $E(n)$ we have an embedded sphere
transversely intersecting a cusp fiber in a single point. The
neighbourhood $N_c$ of the cusp point is modelled by the cone on the
trefoil knot. The same local model occurs near the singular point of
the cuspidal cubic curve
\[
C=\{ [x:y:z]\in {\mathbb {CP}}^2\mid x^3=y^2z\} .
\]
Therefore we can glue $E(n)\setminus N_c$ with the complement of the
similar neighbourhood of the cusp point of $C$ --- as we use an
orientation reversing diffeomorphism between the boundaries, we find a
sphere in $E(n)\# {\overline {\mathbb {CP}}}^2$, the blow-up of
$E(n)$. The resulting sphere has self-intersection $-9$, still
transversely intersecting the sphere in $E(n)$ in a unique point.

For $X_{n,k}=E(n)\# k{\overline {\mathbb {CP}}}^2$ we can follow a
mixed strategy of constructing spheres: start with the monodromy
presentation of $E(n)$ as $(ab)^{6n}$, separate a few cusp (or type
III or IV) fibers (occupying some of the monodromy) for the blow-ups,
partition the rest of the blow-ups to get ${\widetilde {E}}_8$- or
${\widetilde {E}}_6$-fibers (taking $(ab)^5$ or $(ab)^4$ in the
monodromy), blow up the singular points of the cusp (or type III or
IV) fibers sufficiently many times, and apply the method of
Proposition~\ref{prop:twosphereblow}
for the leftover blow-ups.

As examples, we construct spheres in $E(2)\# {\overline {\mathbb
    {CP}}}^2$, in $E(6)\# {\overline {\mathbb {CP}}}^2$, and in
$E(6)\# 3{\overline {\mathbb {CP}}}^2$ using one or more of the
strategies outlined above.

\begin{exa}[Spheres in $E(2)\# {\overline {\mathbb {CP}}}^2$]
\label{exa:k3egy}
  Considering the monodromy of $E(2)$ as $(ab)^5\cdot (ab)^5\cdot (ab)^2$,
  we get an elliptic fibration on the $K3$ surface with three singular
  fibers, two ${\widetilde {E}}_8$ and one type IV. Blowing up the
  singular point of the type IV fiber, we get a tree of spheres with 23
  vertices, 19 of them $(-2)$-spheres, one $(-1)$ and three $(-3)$.
  A simple computation shows the existence of a sphere of self-intersection
  $-92$. (This construction gives a slightly better result than blowing up
  an edge in the tree of spheres giving the $(-86)$-sphere in the $K3$ surface,
  which results a $(-91)$-sphere.)
  \end{exa}

\begin{exa}[Spheres in $E(6)\# {\overline {\mathbb {CP}}}^2$]
\label{exa:e6egy}
  As it was shown in Theorem~\ref{thm:negspheres}, there is a sphere
  of self-intersection $-262$ in $E(6)$; its connected sum with the
  $(-4)$-sphere in ${\overline {\mathbb {CP}}}^2$ gives a
  $(-266)$-sphere. As the $(-262)$-sphere is given by a tree, hence can be
  viewed as smoothing of two spheres  intersecting transversally once, by
  blowing up the intersection point of these spheres we get a
  $(-267)$-sphere in $E(6)\# {\overline {\mathbb {CP}}}^2$.
  Consider now a slightly different construction: take a fibration on
  $E(6)$ with 7 ${\widetilde {E}}_8$-fibers and a cusp (as the total
  monodromy of $E(6)$ is $(ab)^{36}=((ab)^5)^7\cdot (ab)$). Use the seven
  ${\widetilde {E}}_8$-fibers (together with the section)
  to create a $(-258)$-sphere in $E(6)$ intersecting the unique cusp fiber
  in one transverse point. Now summing it with the cuspical cubic in
  ${\overline {\mathbb {CP}}}^2$ we get a $(-9)$-curve, and smoothing
  the intersection with the appropriate orientations we finally receive a
  sphere in $E(6)\# {\overline {\mathbb {CP}}}^2$ of self-intersection
  $-269$.
\end{exa}

\begin{exa}[Spheres in $E(6)\# 3 {\overline {\mathbb {CP}}}^2$]
\label{exa:e6harom}
By presenting the $(-262)$-sphere in
$E(6)$ with a plumbing graph and blowing up edges there, we get a
$(-277)$-sphere in the three-fold blow-up. There is a further
possibility: consider the fibration with 7 ${\widetilde {E}}_8$-fibers
and a cusp fiber, and blow up the cusp three times. Simple calculation
shows that the resulting sphere has self-intersection $-278$. A better
result can be achieved by taking the $(-269)$-sphere of
Example~\ref{exa:e6egy} in $E(6)\# {\overline {\mathbb {CP}}}^2$ and
blow it up twice (as done in Proposition~\ref{prop:twosphereblow}) to
get a sphere with self-intersection $-279$.
\end{exa}
Notice that all the above examples satisfy
Conjecture~\ref{conj:mostnegsphere}; the ratio of the (negative of
the) self-intersection and $b_2$ is significantly less than 5. So far
we found examples with this ratio being close to 5 only after a
significant amount of blow-ups.


\end{document}